\documentclass[11pt]{article}
\usepackage{amsmath, amsfonts, amssymb,  mathrsfs, color, amsthm, epsfig, color, graphics, mathtools,oubraces, enumerate,authblk,esint,empheq}
  \usepackage[all]{xy}
 \usepackage[latin1]{inputenc}
\newtheorem{theorem}{Theorem}[section]
\newtheorem*{theorem*}{Theorem}
\newtheorem{lemma}{Lemma}[section]
\newtheorem{prop}{Proposition}[section]
\newtheorem{definition}{Definition}[section]

\usepackage{anysize} 
\marginsize{2cm}{2cm}{1.5cm}{1.5cm}

\DeclareMathOperator*{\argmin}{argmin}

\DeclareMathAccent{\mpetito}{\mathalpha}{operators}{23}
\newcommand{\dx}{\,{\rm d}x}
\newcommand{\dy}{\,{\rm d}y}
\newcommand{\ds}{\,{\rm d}s}

\newcommand{\C}{\mathscr C}
\newcommand{\forallt}{\qquad\text{for all }}

\newcommand{\fct}[4]{\arraycolsep=1.4pt\begin{array}{rcl}{#1}&\longrightarrow&{#2}\\{#3}&\longmapsto&{#4}\end{array}}
\newcommand{\varGammae}{{\varGamma_{\!e}}}

\providecommand{\keywords}[1]
{
  \small	
  \textbf{\textit{Keywords---}} #1
}

\title{Invertibility criteria for the biharmonic  single-layer potential}
\begin{document}
%
\date{\today}
\author{Alexandre Munnier\footnote{alexandre.munnier@univ-lorraine.fr}}

\affil{Université de Lorraine, CNRS, Inria, IECL, F-54000 Nancy, France}
\maketitle
\begin{abstract}
While the single-layer operator for the Laplacian is well understood, questions remain concerning the single-layer operator for the Bilaplacian, particularly with regard to invertibility issues linked with degenerate scales. In this article, we provide
simple sufficient conditions ensuring this invertibility 
for a wide range of problems.
\end{abstract}
\noindent\keywords{Biharmonic single-layer potential, biharmonic equation, degenerate scales.}


\section{Introduction}
\label{sec:1}
Let $\varGamma$ be a smooth curve in the plane (we don't intend to be rigorous at this stage). For 
any density $q\in H^{-1/2}(\varGamma)$, the harmonic single-layer potential is defined by:
$$S_\varGamma q(x)=\int_\varGamma g_0(x-y)q(y)\ds(y)\forallt x\in\mathbb R^2,$$
where $g_0$ is the fundamental solution of the Laplacian, that reads ($\kappa$ is a positive parameter):
$$g_0(x)=-\frac{1}{2\pi}\ln \frac{|x|}{\kappa}\forallt x\in\mathbb R^2\setminus\{0\}.$$
The operator $S_\varGamma$ is bounded from $H^{-1/2}(\varGamma)$ into $H^1_{\ell oc}(\mathbb R^2)$ and so is the operator:
$$\fct{V_\varGamma:H^{-1/2}(\varGamma)}{H^{1/2}(\varGamma)}{q}{\gamma_\varGamma^D \circ S_\varGamma q,}$$
where $\gamma_\varGamma^D$ stands for the usual Dirichlet trace operator on $\varGamma$. It is well known that $V_\varGamma$ is 
invertible if and only if $\kappa\neq {\rm Cap}_\varGamma$, where  ${\rm Cap}_\varGamma$ is a constant 
called the logarithmic capacity of $\varGamma$. Furthermore, according to \cite[Theorem 8.16]{McLean:2000aa}, if $\kappa>{\rm Cap}_\varGamma$, the operator $S_\varGamma$ is positive definite on $H^{-1/2}(\varGamma)$ and it has one negative eigenvalue if $\kappa 
<{\rm Cap}_\varGamma$. The value of ${\rm Cap}_\varGamma$ is difficult to evaluate in  general  (it is explicitly known only for certain geometries such as a disk, an ellipse, a square...). However, one can use the following simple estimate:  ${\rm Cap}_\varGamma\leqslant R$ where 
$R>0$ is the radius 
of any circle that enclosed $\varGamma$. The successful implementation of a BEM is therefore guaranteed by the verification of the criterion: 
$\kappa>R$. This criterion is elementary and applies to a wide range of problems. To our knowledge, no such simple criterion is available for the biharmonic single-layer potential. 
\par
\medskip
In this paper we consider Jordan curves of class $\mathcal C^{1,1}$ in the plane (see the last section for possible generalizations). 
We denote by $\varGamma$ 
a disjoint union 
of a finite number of such curves.
We define 
$\varOmega_\varGamma^-$ the bounded  domain consisting of the points enclosed by at least one curve and  $\varOmega^+_\varGamma$
its unbounded complement (see Fig.~\ref{fig0}). 
 The multi-connected curve $\varGamma$ can be decomposed into $\varGamma_{\!e}$, the boundary shared by $\varOmega_\varGamma^+$
 and $\varOmega_\varGamma^-$, and the Jordan curves included in $\varOmega_\varGamma^-$. On every Jordan curve, we define $n$ the unit normal vector field directed towards the bounded domain enclosed by the curve (and we will stick to this convention throughout the paper).
\begin{figure}[h]
\centerline{\input{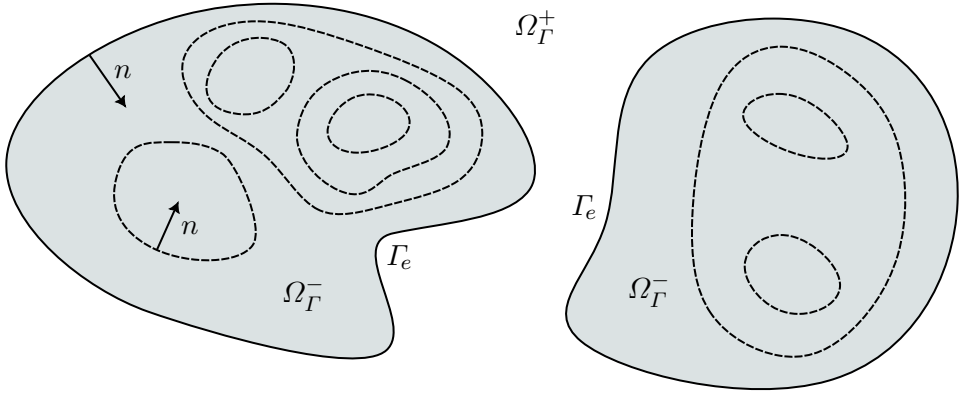}}
\caption{\label{fig0}The multi-connected curve $\varGamma$ can be decomposed into the disjoint union of 
$\varGammae$ and the Jordan curves included in $\varOmega_\varGamma^-$.}
\end{figure}
\par
For every function $u\in H^2_{\ell oc}(\mathbb R^2)$, we can define the Dirichlet and Neumann traces on $\varGamma$ denoted respectively 
by $\gamma_\varGamma^D u$ 
and $\gamma_\varGamma^N u$ (the Neumann trace is defined taking into account the orientation of $n$). The total trace operator is next 
given by:
$$\fct{\gamma_\varGamma:H^2_{\ell oc}(\mathbb R^2)}{H^{3/2}(\varGamma)\times H^{1/2}(\varGamma)}
{u}{\big(\gamma_\varGamma^D u,\gamma_\varGamma^N u\big).}$$
We will use the following expression for the fundamental solution of the Bilaplacian:
\begin{equation}
\label{eq:def_G}
G_0(x)=\frac{1}{8\pi}\Big[|x|^2\ln\frac{|x|}{\kappa_0}+\kappa_1\Big]\forallt x\in\mathbb R^2,
\end{equation}
where the parameters $(\kappa_0,\kappa_1)\in\,]0,+\infty[\times\mathbb R$ are introduced to cover all the classical definitions available in 
the literature (see for instance \cite{Christiansen:1998aa, Constanda:1997aa, Shigeta:2011aa}).
We denote $H(\varGamma)=H^{3/2}(\varGamma)\times H^{1/2}(\varGamma)$ and $H'(\varGamma)=
H^{-3/2}(\varGamma)\times H^{-1/2}(\varGamma)$. Using the usual abuse of notation to identify $G_0(x-y)$ with 
a two-variables function $G_0(x,y)$, the biharmonic single-layer potential  is defined for every $q=(q_0,q_1)\in H'(\varGamma)$ by:
\begin{equation}
\label{eq:def_Sq}
\mathscr S_\varGamma q(x)=\int_\varGamma G_0 (x,y)q_0(y)+\partial_{n(y)}G_0(x,y)q_1(y)\,{\rm d}s(y)\forallt x\in\mathbb R^2.
\end{equation}
The operator $\mathscr S_\varGamma:H'(\varGamma)\longrightarrow H^2_{\ell oc}(\mathbb R^2)$ is bounded
so the same conclusion applies to the operator:
\begin{equation}
\label{def_Vgamma}
\fct{V_\varGamma:H'(\varGamma)}{H(\varGamma)}
{q}{\gamma_\varGamma\circ\mathscr S_\varGamma q.}
\end{equation}
Our goal is to determine conditions that ensure the invertibility of $V_\varGamma$.
Even in the simplest case where $\varGamma$ reduces to a single Jordan curve, the answer is not clear in general. Indeed, for any fixed parameters $\kappa_0$, $\kappa_1$ (in identity \eqref{eq:def_G}), it is known that 
there exist degenerate scales 
$\rho>0$ for which $V_{\rho\varGamma}$ is not invertible. Since the pioneering work \cite{Costabel:1996aa}, studies on degenerate scales (and more generally on the invertibility of $V_\varGamma$) 
have been reduced to questions of invertibility of a $4\times 4$ matrix  (known as the discriminant matrix). In this paper, we show that this matrix
 can be replaced by a simpler one, a $3\times 3$ matrix called the Robin matrix (because of the similarity of its role to that of the Robin constant in potential theory). 
Recalling that 
the operator $V_\varGamma$ (from $H'(\varGamma)$ into itself) is self-adjoint,
our main results are as follows:
\begin{theorem}
\label{main_theo1}
\begin{enumerate}
\item The invertibility of $V_\varGamma$ depends on $\varGammae$ only.
\label{p1}
\item To any multi-connected curve $\varGamma$ (as described above) we can associate a $3\times 3$ symmetric matrix $\Lambda_{\varGammae}$ (the Robin matrix). 
The operator $V_\varGamma$ is an isomorphism if and only if 
$\det \Lambda_\varGammae\neq 0$.
\label{p2}
\item Let $R>0$ be the radius of a circle $\mathcal C_R$ such 
that $\varOmega_{\mathcal C_R}^+\subset \varOmega_\varGamma^+$ (i.e. the circle $\mathcal C_R$ enclosed $\varGamma$). If  $\kappa_0>eR$ and $\kappa_1>(\kappa_0/e)^2$ in the definition \eqref{eq:def_G}, then the 
matrix $\Lambda_{\varGammae}$ is positive definite and the
operator $V_\varGamma$ is strongly elliptic on $H'(\varGamma)$.
\label{p3}
\item
Let $R>0$ be the radius of a circle $\mathcal C_R$ such 
that $\varOmega_\varGamma^+\subset \varOmega_{\mathcal C_R}^+$. If
  $\kappa_0<eR$ and $\kappa_1<(\kappa_0/e)^2$  in the definition \eqref{eq:def_G}, 
then the matrix $\Lambda_{\varGammae}$ is negative definite.
\label{p4}
\item Let $\kappa_0=1$ and $\kappa_1=0$ and let $\mathcal C_{R^-}$ 
and  $\mathcal C_{R^+}$ be  circles of radii $R^-$ and $R^+$ such that $\varOmega_{\mathcal C_{R^+}}^+\subset 
\varOmega_{\varGamma}^+\subset \varOmega_{\mathcal C_{R^-}}^+$. Degenerate scales $\rho$ for $\varGamma$ 
can occur only when $1/(eR^+)<\rho<1/(eR^-)$.
\label{p5}
\end{enumerate}
\end{theorem}
Points~\ref{p3} and \ref{p4} of Theorems~\ref{main_theo1} provide criteria (similar to those for the harmonic single-layer potential) ensuring the successful implementation of a BEM. These criteria 
are simple and apply to most of the problems that can be encountered. In particular, we emphasize that they apply to the case where 
$\varGammae$ is multi-connected, this case being the blind spot of most publications on the subject.
\par
Some authors focus on the particular case where $\kappa_0=1$ and $\kappa_1=0$. The most advanced results concerning 
degenerate scales in this case can be found in the recent paper \cite{Corfdir:2022aa} (to which we refer also for a detailed
overview  of the known results on this topic). The authors look for sufficient conditions 
to prevent the appearance of degenerate scales. They prove that when $\varGamma$ is a single Jordan curve such that the 
domain $\varOmega_\varGamma^-$ is star-shaped, or has symmetry properties, $V_\varGamma$ 
is invertible provided that either a circle of radius $1/e$ is included in $\varOmega_\varGamma^-$ or that 
$\varOmega_\varGamma^-$ is included in  such a circle. They claim that when $\varGammae$ is not connected, 
no general conclusion can be drawn. Points~\ref{p3}, \ref{p4}  and \ref{p5} of Theorem~\ref{main_theo1} above
extend their results to the general case considered in this article, and thus removes their geometric restrictions. 
In the same paper  \cite{Corfdir:2022aa}, it is showed that when $\varGammae$ is a single 
Jordan curve, ``holes'' have no influence on the degenerate scales. The first assertion of Theorem~\ref{main_theo1} also extends 
this result, in particular to the case where $\varGammae$ is multi-connected.
 \par
 The proof of Theorem~\ref{main_theo1} results straightforwardly from the combination of Theorems~\ref{main:theo2}, \ref{main:3}, \ref{the:2} 
 and Property~\ref{prop:1} established below.
\section{The biharmonic transmission problem}
We continue to use the notation $\varGamma$ to designate a disjoint union of Jordan curves as described in the previous section.
Following an idea developed in the article \cite{Amrouche:1994aa}, we introduce the weight functions:
$$\rho(x)=\sqrt{1+|x|^2}\qquad\text{and}\qquad {\rm lg}(x)=\ln(2+|x|^2)\forallt x\in\mathbb R^2,$$
and the weighted Sobolev space:
$$W^2(\mathbb R^2)=\Big\{u\in \mathscr D'(\mathbb R^2)\,:\, \frac{u}{\rho^2\,{\rm lg}}\in L^2(\mathbb R^2),\, \frac{1}{\rho\,{\rm lg}}\frac{\partial u}{\partial x_j}\in L^2(\mathbb R^2)~\text{ and }~\frac{\partial^2 u}{\partial x_j\partial x_k}\in L^2(\mathbb R^2),\,\forall\,j,k=1,2\Big\}.$$
Note that the three-dimensional space of affine functions is  a subspace of $W^2(\mathbb R^2)$.  
The affine functions will play a particular role in the analysis (the same role as played by the constants for the harmonic 
single-layer  potential). 
%
We provide the space $W^2(\mathbb R^2)$ with the inner product:
\begin{equation}
\label{p_scal}
(u,v)_{W^2(\mathbb R^2)}=(\Delta u,\Delta v)_{L^2(\mathbb R^2)}+\int_\varGamma u\,v\ds\forallt u,v\in W^2(\mathbb R^2).
\end{equation}
According to \cite{Amrouche:1994aa}, the norm associated to this scalar product is equivalent to the natural norm of $W^2(\mathbb R^2)$. 
We introduce the subspace:
$$W^2_\varGamma(\mathbb R^2)=\Big\{u\in W^2(\mathbb R^2)\,:\,\gamma_\varGamma  u=0\Big\},$$
and for every $p=(p_0,p_1)\in H(\varGamma)$, we define: 
\begin{equation}
\label{def_SP}
\mathsf S_\varGamma p=\argmin\Big\{\|u\|_{W^2(\mathbb R^2)}\,:\,u\in W^2(\mathbb R^2),\,\gamma_\varGamma u=p\Big\}.
\end{equation}
Being the orthogonal projection of the origin onto the affine space $\mathcal W(p)=\big\{u\in W^2(\mathbb R^2)\,:\gamma_\varGamma u=p\big\}$ (which 
is closed an convex) in the Hilbert space $W^2(\mathbb R^2)$, the minimum is indeed achieved and unique. 
\par
For any function $u$ defined in $\mathbb R^2$ we denote by $u^+$  and $u^-$ its restrictions to the domains $\varOmega_\varGamma^+$ and $\varOmega_\varGamma^-$ respectively. Some elementary properties of $\mathsf S_\varGamma p$ are gathered in the following lemma:
\begin{lemma}
\label{gtqpllkj}
\begin{enumerate}
\item The function $\mathsf S_\varGamma:H(\varGamma)\longrightarrow W_\varGamma^2(\mathbb R^2)^\perp$ is an isomorphism, 
its inverse being $\gamma_\varGamma:W_\varGamma^2(\mathbb R^2)^\perp\longrightarrow H(\varGamma)$.
\item The function $\mathsf S_\varGamma p$ is the unique (weak) solution in the space $W^2(\mathbb R^2)$ to the 
transmission problem:
\begin{subequations}
\label{uniq_trans}
\begin{empheq}[left=\empheqlbrace]{align}
\Delta^2 u &=0\quad\text{in }\mathbb R^2\setminus\varGamma\\
\gamma_\varGamma  u&=p.
\end{empheq}
\end{subequations}
%
%
\item For every $u\in W^2_\varGamma(\mathbb R^2)^\perp$  and $v\in W^2(\mathbb R^2)$, $(\Delta u,\Delta v)_{L^2(\varOmega_\varGamma^+)}=\int_{\varGammae} (\partial_n \Delta u^+)v-(\Delta u^+)\partial_n v\ds$.
\item\label{frlopwwk1}
If $\mathscr C$ is another set of Jordan curves as described in Section~\ref{sec:1} such that $\varOmega_\C^+\subset 
\varOmega_\varGamma^+$, then for every $p\in H(\varGamma)$, 
$\mathsf S_{\C_e}\circ\gamma_{\C_e}\circ\mathsf S_\varGamma p=\mathsf S_\varGamma p$ in $\varOmega^+_{\C}$.
\item For every $p\in H(\varGamma)$, $\Delta\mathsf S_\varGamma p(x)=\mathscr O(1/|x|)$ as $|x|\longrightarrow +\infty$.
\end{enumerate}
\end{lemma}
\begin{proof}
\begin{enumerate}
\item As already mentioned, for any $p\in H(\varGamma)$, $\mathsf S_\varGamma p$ is the orthogonal projection of the origin onto the 
affine subspace $\mathcal W(p)$ of $W^2(\mathbb R^2)$. It is well known that this orthogonal projection can be defined equivalently as the only function of  
$\mathcal W(p)$ such 
that $\big(\mathsf S_\varGamma p-0,u-v\big)_{W^2(\mathbb R^2)}=0$ for all $u,v\in \mathcal W(p)$, i.e. the only function 
such that
\begin{equation}
\label{frgpp}
\big(\mathsf S_\varGamma p,w\big)_{W^2(\mathbb R^2)}=0\forallt w\in W_\varGamma^2(\mathbb R^2),
\end{equation}
which  proves that the function  $\mathsf S_\varGamma$ is indeed valued in $W_\varGamma^2(\mathbb R^2)^\perp$. 
It is now straightforward to verify that its inverse is $\gamma_\varGamma:W_\varGamma^2(\mathbb R^2)^\perp\longrightarrow H(\varGamma)$.
\item Choosing $w\in\mathscr D(\mathbb R^2\setminus\varGamma)$ in \eqref{frgpp}, we deduce that $\Delta^2\mathsf S_\varGamma p=0$ 
in $\mathscr D'(\mathbb R^2\setminus\varGamma)$ for all $p\in H(\varGamma)$ and therefore that $\mathsf S_\varGamma p$ is a weak 
solution of \eqref{uniq_trans}. Reciprocally, any weak solution $u$ of \eqref{uniq_trans} satisfies $(u,w)_{W^2(\mathbb R^2)}=0$ 
for every $w\in\mathscr D(\mathbb R^2\setminus\varGamma)$. From \cite[Theorem 7.2]{Amrouche:1994aa} we deduce that 
$\mathscr D(\mathbb R^2\setminus\varGamma)$ is dense in $W^2_\varGamma(\mathbb R^2)$, which entails that $u$ satisfies 
\eqref{frgpp} and hence $u=\mathsf S_\varGamma p$.
\item Since $\Delta u^+$ is  a harmonic function in $L^2(\varOmega^+_\varGamma)$, 
it admits a Dirichlet and a Neumann trace in $H^{-1/2}(\varGammae)$ and $H^{-3/2}(\varGammae)$ respectively. The integration by parts formula results from the density of the space $\mathscr D(\mathbb R^2)$ 
in $W^2(\mathbb R^2)$ (asserted in \cite[Theorem 7.2]{Amrouche:1994aa}).
\item We get the result by noticing that the function equal to $\mathsf S_{\C_e}\circ\gamma_{\C_e}\circ\mathsf S_\varGamma p-
\mathsf S_\varGamma p$ in $\varOmega_{\C}^+$ and to $0$ in $\varOmega_{\C}^-$, is in  $W^2_{\C_e}(\mathbb R^2) \cap W^2_{\C_e}(\mathbb R^2)^\perp$. 
\item For every $p\in H(\varGamma)$ and every $x\in\varOmega_\varGamma^+$, the mean value property for harmonic functions 
asserts that:
$$\Delta \mathsf S_\varGamma p(x)=\frac{1}{\pi R_x}\int_{D(x,R_x)}\Delta  \mathsf S_\varGamma p(y)\dy,$$
where $D(x,R_x)$ is the disk of center $x$ and radius $R_x$ with $R_x$  the distance from $x$ to $\varGamma$. Since $\Delta  \mathsf S_\varGamma p\in L^2(\varOmega_\varGamma^+)$, 
the result follows from Cauchy-Schwarz inequality.
\end{enumerate}
\end{proof}
It should be noted here that $\mathsf S_\varGamma p$ is not the biharmonic single-layer potential of total trace $p$. For instance, if $p$ is the total trace of an affine function, then 
$\mathsf S_\varGamma p$ is equal to this function while the biharmonic single-layer potential is not. 
\section{The Robin matrix}
\label{sec:3}
In this section, we will define the Robin matrix of a multi-connected curve $\varGamma$. Since the Robin matrix depends only on 
$\varGammae$, to lighten the notation, we will assume that $\varGamma$ is such that $\varGamma=\varGammae$. We assume also 
that the origin lies in $\varOmega_\varGamma^-$ and we introduce the functions:
$$G_1(x)=-\frac{\partial G_0}{\partial x_1}(x)=-\frac{x_1}{8\pi}\Big[ 2\ln\frac{|x|}{\kappa_0}+1\Big]\qquad
\text{and}\qquad G_2(x)=-\frac{\partial G_0}{\partial x_2}(x)=-\frac{x_2}{8\pi}\Big[ 2\ln\frac{|x|}{\kappa_0}+1\Big],$$
and therefore, denoting by $\omega_j$ the Laplacian of $G_j$ (for $j=0,1,2$), we have:
$$\omega_0(x)=\frac{1}{2\pi}\Big[\ln\frac{|x|}{\kappa_0}+1\Big],\qquad\omega_1(x)= -\frac{1}{2\pi}\frac{x_1}{|x|^2},\qquad \omega_2(x)
= -\frac{1}{2\pi}\frac{x_2}{|x|^2}\forallt x\in\mathbb R^2\setminus\{0\}.$$
The following notations will also be helpful:
$$G(x)=\begin{pmatrix}
G_0(x)\\G_1(x)\\G_2(x)\end{pmatrix}\qquad\text{and}\qquad X(x)=\begin{pmatrix}1\\x_1\\x_2\end{pmatrix}\qquad\text{and}\qquad 
\omega(x)=\begin{pmatrix}\omega_0(x)\\ \omega_1(x)\\ \omega_2(x)\end{pmatrix}.$$
Let $u$ be a biharmonic function in $H^2_{\ell oc}(\overline{\varOmega_\varGamma^+})$ and $p$ be in $H(\varGamma)$, 
and define:
\begin{equation}
\label{croch}
\big[u,p\big]_\varGamma=-\int_{\varGamma}\partial_n(\Delta u^+)\mathsf S_\varGamma p\ds+\int_{\varGamma}(\Delta u^+)\partial_n \mathsf S_\varGamma p\ds
-\int_\varGamma\partial_n u(\Delta \mathsf S^+_\varGamma p)\ds+\int_\varGamma  u\,\partial_n(\Delta \mathsf S^+_\varGamma p)\ds.
\end{equation}
This bracket will prove crucial in the analysis. Let us collect some of its properties:
\begin{lemma}
\label{wppfrqe}
\begin{enumerate}
\item For every $p,q\in H(\varGamma)$, $\big[\mathsf S_\varGamma p, q\big]_\varGamma=0$.
\item If $\mathscr C$ is another Jordan curve  such that $\varOmega_\C^+\subset 
\varOmega_\varGamma^+$, then for every $p\in H(\varGamma)$ and $j=0,1,2$, $\big[G_j,p\big]_\varGamma=
\big[G_j,\gamma_{\C}\circ\mathsf S_\varGamma p\big]_{\C}$.
\item For $j,k=0,1,2$, $\big[G_j,(\gamma_\varGamma X_k)\big]_\varGamma= \delta_{jk}$.
\item For $j,k=0,1,2$, $\big[G_j, \gamma_\varGamma G_k\big]_\varGamma=\big[G_k, \gamma_\varGamma G_j\big]_\varGamma$.
\end{enumerate}
\end{lemma}
\begin{proof}\begin{enumerate}
\item It suffices to combine the definition \eqref{croch} with the third point of Lemma~\ref{gtqpllkj}.
\item An integration by parts in the domain between $\varGamma$ and $\C$ followed by the fourth point of Lemma~\ref{gtqpllkj} leads to 
the equality.
\item Let $\mathcal C_R$ be a large circle of radius $R$, centered at the origin and enclosing $\varGamma$. The preceding point asserts that:
\begin{equation}
\label{wxpghyi}
\big[G_j, (\gamma_\varGamma X_k)\big]_\varGamma=\big[G_j, (\gamma_{\mathcal C_R} X_k)\big]_{\mathcal C_R}=-\int_{\mathcal C_R}(\partial_n \omega_j) X_k \ds+
\int_{\mathcal C_R}\omega_j (\partial_n X_k)\ds.
\end{equation}
Then, explicit computations yield the result.
\item 
According to the third point of Lemma~\ref{gtqpllkj}:
\begin{subequations}
\label{fvprwo}
\begin{multline}
\int_\varGamma(\partial_n G_j) (\Delta \mathsf S^+_\varGamma\circ\gamma_\varGamma G_k)\ds-
\int_\varGamma G_j\partial_n(\Delta \mathsf S^+_\varGamma\circ\gamma_\varGamma G_k)\ds=\\
\int_\varGamma(\partial_n G_k) (\Delta \mathsf S^+_\varGamma\circ\gamma_\varGamma G_j)\ds-
\int_\varGamma G_k\partial_n(\Delta \mathsf S^+_\varGamma\circ\gamma_\varGamma G_j)\ds.
\end{multline}
On the other hand, let $D$ be the domain between $\varGamma$ and a large circle $\mathcal C_R$ centered at the origin. 
Integrating by parts the zero quantity $(\Delta\omega_j,G_k)_{L^2(D)}$, we get:
\begin{equation}
\int_\varGamma(\partial_n \omega_j)G_k\ds-
\int_\varGamma \omega_j(\partial_n G_k)\ds=\int_{\mathcal C_R}(\partial_n \omega_j) G_k\ds-
\int_{\mathcal C_R} \omega_j(\partial_n G_k)\ds+(\omega_j,\omega_k)_{L^2(D)},
\end{equation}
\end{subequations}
and one easily verifies that the boundary integrals in the right hand side vanish  when $j\neq k$.
Using the equalities \eqref{fvprwo} in the definition of $\big[G_j,\gamma_\varGamma G_k\big]_\varGamma$ 
leads to the result. 
\end{enumerate}
\end{proof}
For every $p\in H(\varGamma)$ we define $\big[G,p\big]_\varGamma$ as the vector in $\mathbb R^3$ whose components are 
$\big[G_j,p\big]_\varGamma$ ($j=0,1,2$).
\begin{definition}
The $3\times 3$ matrix:
$$\Lambda_\varGamma=\big(\big[G_j,\gamma_\varGamma G_k\big]_\varGamma\big)_{0\leqslant j\leqslant 2\atop 0\leqslant k\leqslant 2},$$
will be called the Robin matrix (by analogy with the Robin  constant for the harmonic single-layer potential).
The fourth point of Lemma~\ref{wppfrqe} asserts that this matrix is symmetric. 
\end{definition}
We define now, for $k=0,1,2$:
\begin{equation}
\label{defGbi}
\mathscr G^k_\varGamma(x)=\begin{cases}
G_k(x)-\mathsf S_\varGamma\circ\gamma_\varGamma G_k(x)
+\big[G,\gamma_\varGamma G_k\big]_\varGamma\cdot X(x)&(x\in\varOmega^+_\varGamma)\\
\big[G,\gamma_\varGamma G_k\big]_\varGamma\cdot X(x)&(x\in\varOmega^-_\varGamma)
\end{cases}
\qquad\text{and}\qquad
\mathscr G_\varGamma=\begin{pmatrix}\mathscr G^0_\varGamma\\\mathscr G^1_\varGamma\\\mathscr G^2_\varGamma\end{pmatrix}.
\end{equation}
Applying the total trace operator to each component of the vectors, we obtain the equality:
\begin{equation}
\label{wxgaqse}
\gamma_\varGamma \mathscr G_\varGamma=\Lambda_\varGamma \big(\gamma_\varGamma X\big).
\end{equation}
\section{Invertibility of the biharmonic single-layer potential}
We now return to the general case in which $\varGamma$ represents a union of Jordan curves, as described in Section~\ref{sec:1}. Up to a translation, we can assume that the origin lies in $\varOmega_\varGamma^-$. For every $p\in H(\varGamma)$, we denote by $p_e$ the restriction of $p$ to $\varGammae$.
\begin{definition}
If $\det \Lambda_{\varGammae}\neq 0$, we define for every $p\in H(\varGamma)$:
\begin{equation}
\label{def_Sp2}
\mathscr S^\dagger_\varGamma p=\mathsf S_\varGamma p-\big[G,p_e\big]_{\varGammae}\cdot X+
\big[G,p_e\big]_{\varGammae}\cdot\Lambda_{\varGammae}^{-1}\mathscr G_{\varGammae}.
\end{equation}
\end{definition}
We are going to show that, this time, $\mathscr S^\dagger_\varGamma p$ coincides well 
with the biharmonic single-layer potential of total trace $p$. 
\par
Every function $u$ in $L^2_{\ell oc}(\mathbb R^2)$ harmonic in $\mathbb R^2\setminus\varGamma$ admits
one-sided Dirichlet and Neumann traces on $\varGamma$ in the spaces $H^{-1/2}(\varGamma)$ 
and $H^{-3/2}(\varGamma)$ respectively. Taking into account the orientation of the unit normal $n$, we set:
\begin{equation}
\label{conv}
\big[u\big]_\varGamma=\gamma_\varGamma^Du^+-\gamma_\varGamma^D u^-\qquad\text{and}\qquad 
\big[\partial_n u\big]_\varGamma=\gamma_\varGamma^Nu^+-\gamma_\varGamma^N u^-,
\end{equation}
and we define the operator:
\begin{equation}
\fct{U_\varGamma:H(\varGamma)}{H'(\varGamma)}{p}{\big(-\big[\partial_n \Delta \mathscr S_\varGamma^\dagger p\big]_\varGamma,
\big[\Delta \mathscr S_\varGamma^\dagger p\big]_\varGamma\big).}
\end{equation}
\begin{theorem}
\label{main:theo2}
The operator $V_\varGamma$ is invertible if and only if $\det \Lambda_{\varGammae}\neq 0$. In this case
$\mathscr S^\dagger_\varGamma p=\mathscr S_\varGamma\circ U_\varGamma p$ for every $p\in H(\varGamma)$. Taking the total trace on $\varGamma$, this entails that $V_\varGamma\circ U_\varGamma p=p$.
\end{theorem}
The proof is based on a couple of technical lemmas. The first is the result of explicit calculations:
\begin{lemma}
For any $q=(q_0,q_1)\in H'(\varGamma)$, the biharmonic single-layer potential $\mathscr S_\varGamma q$ (defined by \eqref{eq:def_Sq})
and its partial derivatives up to order 2 admit the following asymptotic expansions as $|x|$ goes to $+\infty$:
\begin{subequations}
\label{asymp_SLP}
\begin{align}
\label{asymp_SLP1}
\mathscr S_\varGamma q(x)&=A_\varGamma(q)\cdot G(x)
+B_\varGamma(q)\omega_0(x)+
C_\varGamma(q)\frac{x_1^2-x_2^2}{|x|^2}+
D_\varGamma(q)\frac{x_1x_2}{|x|^2}+
\mathscr O(1/|x|),\\[-1mm]
\partial_{x_j}\mathscr S_\varGamma q(x)&=A_\varGamma(q)\cdot \partial_{x_j}G(x)
+\mathscr O(1/|x|),\\[1mm]
\partial^2_{x_jx_k}\mathscr S_\varGamma q(x)&=A_\varGamma(q)\cdot \partial_{x_jx_k}^2G(x)
+\mathscr O(1/|x|^2).
\label{asymp_SLP3}
\end{align}
\end{subequations}
where $A_\varGamma(q)\in\mathbb R^3$ is defined by:
$$A_\varGamma(q)=\int_\varGamma q_0(y)X(y)+q_1\partial_n X(y)\ds,$$
and $B_\varGamma(q)$, $C_\varGamma(q)$ and $D_\varGamma(q)$ are real constants depending on $q$.
\end{lemma}
\begin{lemma}
\label{prop4}Assume that $\det \Lambda_{\varGammae}\neq 0$.
Then for every $p\in H(\varGamma)$ , $A_\varGamma\big(U_\varGamma p\big)=\Lambda_{\varGammae}^{-1}\big[G,  p_e\big]_{\varGammae}$.
\end{lemma}
\begin{proof}
By definition, we have:
$$A_\varGamma\big(U_\varGamma p\big)=-\int_\varGamma\big[\partial_n(\Delta \mathscr S^\dagger_\varGamma p)\big]_\varGamma\,X\ds
+\int_\varGamma\big[\Delta \mathscr S^\dagger_\varGamma p\big]_\varGamma \partial_nX\ds,$$
and since the functions $X_j$ (for $j=0,1,2$) are harmonic in $\mathbb R^2$, an integration by parts on the domain $\varOmega_\varGamma^-$ yields:
$$A_\varGamma\big(U_\varGamma p\big)=-\int_\varGammae\partial_n(\Delta \mathscr S^\dagger_\varGamma p)^+\,X_j\ds+\int_\varGammae(\Delta \mathscr S^\dagger_\varGamma p)^+ \partial_nX\ds=\big[ \mathscr S^\dagger_\varGamma p,\gamma_\varGammae X\big]_\varGammae.$$
Using the definition \eqref{def_Sp2} of $\mathscr S^\dagger_\varGamma p$, the fourth point of Lemma~\ref{gtqpllkj} and the first point of 
Lemma~\ref{wppfrqe}, we obtain:
$$\big[ \mathscr S^\dagger_\varGamma p,\gamma_\varGammae X\big]_\varGammae=\big[
\big[G,p_e\big]_\varGammae\cdot\Lambda_\varGammae^{-1}G,\gamma_\varGammae X\big]_\varGammae.$$
Then, the third point of Lemma~\ref{wppfrqe} leads to the conclusion.
\end{proof}
\begin{proof}[Proof of Theorem~\ref{main:theo2}]Assume that $\det \Lambda_{\varGammae}\neq 0$, let $p$ be in $H(\varGamma)$ and define $u=\mathscr S^\dagger_\varGamma p-\mathscr S_\varGamma\circ U_\varGamma p$. By construction, $\big[\Delta u\big]_\varGamma=0$ and $\big[\partial_n\Delta u\big]_\varGamma=0$, 
hence the function $\Delta u$ is harmonic in $\mathbb R^2$. Moreover, according to the last assertion of Lemma~\ref{gtqpllkj}:
\begin{equation}
\label{aplgyuij}
\Delta \mathscr S^\dagger_\varGamma p(x)=\big[G,p_e\big]_\varGammae\cdot\Lambda_\varGammae^{-1}\omega(x)+\mathscr O(1/|x|)
\qquad\text{as}\qquad |x|\longrightarrow+\infty.
\end{equation}
Combining \eqref{asymp_SLP3}, \eqref{aplgyuij} and Lemma~\ref{prop4}, we deduce that $\Delta u(x)$ tends to 0 as $|x|$ goes to $+\infty$ and
Liouville's theorem allows us to conclude that $\Delta u$ is equal to zero. Therefore, 
there exists a function $h$ harmonic in $\mathbb R^2$ such that $\mathscr S^\dagger_\varGamma p= \mathscr S_\varGamma \circ U_\varGamma p+h$
and hence also:
\begin{equation}
\label{dlkfjg}
\big(\mathscr S^\dagger_\varGamma p-\big[G,p_e\big]_\varGammae\cdot  \Lambda_\varGammae^{-1} G\big)=
 \big(\mathscr S_\varGamma \circ U_\varGamma p-\big[G,p_e\big]_\varGammae\cdot  \Lambda_\varGammae^{-1} G \big)+h\qquad
 \text{in }\varOmega_\varGamma^+.
 \end{equation}
Denote  by $p_1$ and $p_2$ respectively the total traces on $\varGammae$ of the functions in brackets.  
Considering again their asymptotic behavior, we deduce that they are equal to $\mathsf S_\varGammae p_1$ and $\mathsf S_\varGammae p_2$
respectively, in $\varOmega_\varGamma^+$. On the one hand, it follows that $h$ is in the space $W^2(\mathbb R^2)$ and therefore that $h=\alpha\cdot X$ with $\alpha\in\mathbb R^3$ (the affine functions are the only functions harmonic in $\mathbb R^2$ in $W^2(\mathbb R^2)$). On the other hand, since $p_1=p_e-\gamma_\varGammae\big[G,p_e\big]_\varGammae\cdot  \Lambda_\varGammae^{-1} G$, we have $\big[G_k,  p_1\big]_\varGammae=0$  and 
 $\big[G_k,\alpha\cdot (\gamma_\varGammae X)\big]_\varGammae=\alpha_k$ for $k=0,1,2$, according to the third point of Lemma~\ref{wppfrqe}. We are now going to verify that $\big[G_k,  p_2\big]_\varGammae=0$, which together with equality \eqref{dlkfjg} 
will allow us to conclude that $h=0$.

Let $\mathcal C_R$ be a large circle of radius $R$ and centered at the origin enclosing $\varGammae$. According 
to the second point of Lemma~\ref{wppfrqe}:
$$\big[G_k,p_2\big]_\varGammae=\big[G_k,\gamma_{\mathcal C_R}\circ\mathsf S_\varGammae p_2\big]_{\mathcal C_R}.$$
Define now:
\begin{subequations}
\label{def_mu_lam}
\begin{alignat}{3}
\nu_0(R)&=\frac{R^2}{4}\Big[2\ln \Big(\frac{R}{\kappa_0}\Big)+1\Big]&\qquad&
\qquad \lambda_0(R)=-\frac{R^2}{4\pi}
\Big[\ln\Big(\frac{R}{\kappa_0}\Big)+\ln^2\Big(\frac{R}{\kappa_0}\Big)\Big]+\frac{1}{8\pi}\big[\kappa_1-R^2\big],\\
\nu_j(R)&=-\frac{R^2}{4}&&
\qquad \lambda_j(R)=-\frac{1}{4\pi}\Big[\ln\Big(\frac{R}{\kappa_0}\big)+1\Big],\qquad(j=1,2),
\end{alignat}
\end{subequations}
and the functions $G_k^\dagger=\nu_k(R)\omega_k+\lambda_k(R)X_k$. Then $\gamma_{\mathcal C_R}G_k=\gamma_{\mathcal C_R}G_k^\dagger$ (the total traces coincide on 
$\mathcal C_R$), and since the function $G_k^\dagger$ is harmonic in $\varOmega_{\mathcal C_R}^+$, we deduce, using the third point of Lemma~\ref{gtqpllkj} that:
\begin{equation}
\label{ppppolp}
\int_{\mathcal C_R}\partial_n G_k(\Delta \mathsf S_\varGammae p_2)\ds-\int_{\mathcal C_R}  G_k\,\partial_n(\mathsf S_\varGammae p_2)\ds
=\int_{\mathcal C_R}\partial_n G^\dagger_k(\Delta \mathsf S_\varGammae p_2)\ds-\int_{\mathcal C_R}  G^\dagger_k\,\partial_n(\Delta\mathsf S_\varGammae p_2)\ds=0.
\end{equation}
Using the expansion \eqref{asymp_SLP1} for $\mathsf S_\varGammae p_2$ (recall that $\mathsf S_\varGammae p_2=\mathscr S_\varGamma \circ U_\varGamma p-A_\varGamma(U_\varGamma p))\cdot G$ in $\varOmega_\varGamma^+$), we arrive at:
$$\big[G_k,\gamma_{\mathcal C_R}\circ\mathsf S_\varGammae p_2\big]_{\mathcal C_R}=-\int_{\mathcal C_R}(\partial_n \omega_k) (\mathsf S_\varGammae p_2) \ds+
\int_{\mathcal C_R} \omega_k (\partial_n \mathsf S_\varGammae p_2)\ds\longrightarrow 0\qquad\text{as}\qquad R\longrightarrow+\infty.$$ 
All together, we have proved that $\big[G_k,p_2\big]_\varGammae=0$ for $k=0,1,2$ and therefore that $u=\mathscr S^\dagger_\varGamma p-\mathscr S_\varGamma\circ U_\varGamma p=0$.
\par
Assume now that $\det \Lambda_{\varGammae}= 0$. According to the definition \eqref{defGbi}, this implies that there exists 
$\xi\in\mathbb R^3$, $\xi\neq 0$ such that the function $\mathcal G=\mathscr G_{\varGammae}\cdot\xi$ is zero in $\varOmega_\varGamma^-$.
Let $q=\big(-\big[\partial_n\Delta \mathcal G\big]_\varGamma,\big[\Delta \mathcal G\big]_\varGamma\big)$ and $v=\mathcal G- \mathscr S_\varGamma q$.
We are going to verify that $v=0$ and hence that $V_\varGamma$ is not injective. Since $\mathcal G$ vanishes in $\varOmega_\varGamma^-$, 
we have:
\begin{equation}
\label{aab}
A_\varGamma(q)=-\int_\varGamma\big[\partial_n(\Delta \mathcal G)\big]_\varGamma\,X\ds
+\int_\varGamma\big[\Delta \mathcal G \big]_\varGamma \partial_nX\ds=
-\int_\varGammae \partial_n(\Delta \mathcal G)^+\,X\ds
+\int_\varGammae (\Delta \mathcal G)^+ \partial_nX\ds.
\end{equation}
On the other hand, combining the first and third points of Lemma~\ref{wppfrqe}, we obtain that:
\begin{equation}
\label{abb}
-\int_\varGammae \partial_n(\Delta \mathscr G^k_{\varGammae})^+\,X_j\ds
+\int_\varGammae (\Delta \mathscr G^k_{\varGammae})^+ \partial_nX_j\ds=\delta_{jk}\forallt j,k=0,1,2.
\end{equation}
Comparing \eqref{aab} and \eqref{abb}, it follows that $A_\varGamma(q)=\xi$ and therefore that $\mathcal G$ and  $\mathscr S_\varGamma q$ have the same asymptotic 
behavior. The rest of the proof is similar to the one establishing that $u=0$ above.
\end{proof}
\begin{theorem}
\label{main:3}
If the matrix $\Lambda_\varGammae$ is positive definite, the operator $V_\varGamma$ is strongly elliptic in $H'(\varGamma)$.
\end{theorem}
\begin{proof}
Assume that the matrix $\Lambda_\varGammae$ is positive definite  and 
define in $H(\varGamma)$ the inner product:
\begin{equation}
\label{prod_scal_2}
(p,q)_{\varGamma}=\big(\Delta\mathsf S_\varGamma p,\Delta\mathsf S_\varGamma q\big)_{L^2(\mathbb R^2)}+
\big[G,  p_e\big]_\varGammae\cdot \Lambda_\varGammae^{-1}\big[G,  q_e\big]_\varGammae\forallt p,q\in H(\varGamma).
\end{equation}
The  norm $\|\cdot\|_\varGamma$ associated 
to this scalar product is clearly equivalent to the usual norm on $H(\varGamma)$.
The inclusion $H(\varGamma)\subset L^2(\varGamma)\times L^2(\varGamma)$ being continuous and dense, we can use $
\mathbf L^2(\varGamma)=L^2(\varGamma)\times L^2(\varGamma)$ as pivot space (identified with its dual) and  obtain a Gelfand triple
of Hilbert spaces:
$$H(\varGamma)\subset \mathbf L^2(\varGamma)\subset 
 H'(\varGamma).$$
With this configuration, it is well known that the operator $H(\varGamma)\longrightarrow H'(\varGamma)$, 
$p\longmapsto ( p,\cdot)_{\varGamma}$
is an isometry and elementary calculations can be used to check that it is equal to $U_\varGamma$. We have therefore:
$$\langle U_\varGamma p,p\rangle=\|p\|^2_{\varGamma}\forallt p \in H(\varGamma),$$%
where the brackets $\langle\cdot,\cdot\rangle$ stands for the duality pairing on $H'(\varGamma)\times H(\varGamma)$ that extends the 
inner product of $\mathbf L^2(\varGamma)$. Since $V_\varGamma$ is the inverse of the isometric operator $U_\varGamma$, 
we get the result.
\end{proof}
\section{Properties of  the Robin matrix}
Since the Robin matrix depends only on $\varGammae$, we assume again in this section (as in Section~\ref{sec:3}) that the multi-connected 
curve $\varGamma$ is such that $\varGamma=\varGammae$. Let $\varGamma'$ be another such a curve
 and denote by $\lambda_\varGamma^1\leqslant \lambda_\varGamma^2\leqslant \lambda_\varGamma^3$  the eigenvalues of $\Lambda_\varGamma$
and by $\lambda_{\varGamma'}^1\leqslant \lambda_{\varGamma'}^2\leqslant \lambda_{\varGamma'}^3$ those of $\Lambda_{\varGamma'}$.
\begin{theorem}
\label{the:2}
If $\varOmega_{\varGamma}^+\subset \varOmega_{\varGamma'}^+$ then $\lambda_\varGamma^j\leqslant \lambda_{\varGamma'}^j$ for $j=1,2,3$.
\end{theorem}
\begin{proof}
Let $\xi$ be in $\mathbb R^3$ and define $F=\xi\cdot G$ and $\mathscr F_{\varGamma'}=\xi\cdot\mathscr G_{\varGamma'}$. 
Then $\xi\cdot \Lambda_{\varGamma'}\xi=
\big[F,\gamma_{\varGamma'}F\big]_{\varGamma'}$ and proceeding as for establishing the equality \eqref{abb}, we obtain:
$$\big[F,\gamma_{\varGamma'}F\big]_{\varGamma'}=-\int_{\varGamma'}\big(\partial_n\Delta 
{\mathscr F}^+_{\varGamma'}\big)F\ds+\int_{\varGamma'}\big(\Delta 
{\mathscr F}^+_{\varGamma'}\big)(\partial_n F)\ds.$$
 From \eqref{defGbi} and recalling \eqref{def_SP} we deduce that the above equality can be transformed into:
$$\big[F,\gamma_{\varGamma'}F\big]_{\varGamma'}=-\int_{\varGamma'}\big(\partial_n\Delta 
F\big)F\ds+\int_{\varGamma'}\big(\Delta 
F\big)(\partial_n F)\ds-\min\big\{\|\Delta u\|^2_{L^2(\varOmega_{\varGamma'}^+)}\,:\,u\in W^2(\mathbb R^2),\,\gamma_{\varGamma'}u=F\big\}.$$
Denote by $D$ the bounded domain between $\varGamma'$ and $\varGamma$. On the one hand:
$$-\int_{\varGamma'}\big(\partial_n\Delta 
F \big)F \ds+\int_{\varGamma'}\big(\Delta 
F \big)(\partial_nF )\ds=-\int_{\varGamma}\big(\partial_n\Delta 
F \big)F \ds+\int_{\varGamma}\big(\Delta 
F \big)(\partial_nF )\ds+\int_D|\Delta F |^2\dx.$$
On the other hand:
$$\min\big\{\|\Delta u\|^2_{L^2(\varOmega_{\varGamma'}^+)}\,:\,u\in W^2(\mathbb R^2),\,\gamma_{\varGamma'}u=F\big\}\leqslant 
\int_D|\Delta F|^2\dx+\min\big\{\|\Delta u\|^2_{L^2(\varOmega_\varGamma^+)}\,:\,u\in W^2(\mathbb R^2),\,\gamma_{\varGamma}u=F\big\},$$
and thus we have proved that $\big[F ,\gamma_{\varGamma'}F \big]_{\varGamma'}\geqslant \big[F ,\gamma_{\varGamma}F \big]_{\varGamma}$. The Courant-Fischer min-max principle leads to the conclusion of the theorem.
\end{proof}
\begin{prop}
\label{prop:1}
Let $\mathcal C_R$ be a circle of radius $R>0$. Then $\Lambda_{\mathcal C_R}={\rm diag}(\lambda_0(R),\lambda_1(R),\lambda_2(R))$
where the expressions of the real numbers $\lambda_j(R)$ ($j=0,1,2$) are given in \eqref{def_mu_lam}.
\end{prop}
\begin{proof}For $k=0,1,2$ we have $\mathsf S_{\mathcal C_R}\circ\gamma_{\mathcal C_R} G_k=\nu_k(R)\omega_j
+\lambda_j(R)X_j$ in $\varOmega_\varGamma^+$, where the expressions 
of the constants $\nu_k(R)$ and $\lambda_k(R)$ are given in \eqref{def_mu_lam}. 
The result follows from the definition of the entries $\big[G_k,\gamma_\varGamma G_j\big]_{\mathcal C_R}$ of the matrix 
$\Lambda_{\mathcal C_R}$.
\end{proof}
\section{Generalization}
The $\mathcal C^{1,1}$ regularity of $\varGamma$ is the weakest for which the total trace operator from $H^2_{\ell oc}(\mathbb R^2)$ 
into $H^{3/2}(\varGamma)\times H^{1/2}(\varGamma)$ is well defined and onto. With weaker regularity, $H^{3/2}(\varGamma)$ has no more intrinsic definition and the space $H(\varGamma)$ must be defined simply as the image of $H^2_{\ell oc}(\mathbb R^2)$ by the total trace operator. 
In general, this image is difficult to characterize (see \cite{Lamberti:2020aa, Lamberti:2022aa}). Nevertheless, some cases are dealt with in the literature: curvilinear $\mathcal C^{1,1}$ polygons in \cite{Grisvard:1985aa} and Lipschitz continuous curves in \cite{Geymonat:2007wx}. 
For $\mathcal C^{1,1}$ curvilinear polygons, 
a generalization of the work done in this document seems well within reach.
\par
Another possible generalization would be to adopt the approach of  \cite{Costabel:1996aa}, in which $\varGamma$ 
is any compact set in $\mathbb R^2$. In this case, the spaces $H(\varGamma)$ and $H'(\varGamma)$ would be replaced by 
the spaces denoted respectively by $H_\gamma^2(\varGamma)$ and $H^{-2}_\varGamma$ in \cite{Costabel:1996aa}.
The space $W^2(\mathbb R^2)$ would obviously remain unchanged and $W^2_\varGamma(\mathbb R^2)$ would 
be defined as the closure of $\mathscr D(\mathbb R^2\setminus\varGamma)$ in $W^2(\mathbb R^2)$. With these settings, the domain $\varOmega_\varGamma^+$ 
is the unbounded connected component of $\mathbb R^2\setminus\varGamma$ and $\varGammae$ is the boundary of 
$\varOmega_\varGamma^+$. The details of the analysis still need to be verified, and this will be done in a future work.
\appendix
\section{Statements and Declarations}
\subsection*{Funding}
No funds, grants, or other support was received. The author has no relevant financial or non-financial interests to disclose.
\subsection*{Data availability}
Data sharing not applicable to this article as no datasets were generated or analysed during the current study.

\end{document}